\documentclass[11pt]{article}
\usepackage[ansinew]{inputenc}
\usepackage{graphicx}
\usepackage{color}
\usepackage[colorlinks]{hyperref}

\usepackage{enumerate,latexsym}
\usepackage{latexsym}
\usepackage{amsmath,amssymb}% CUIDADO: Si uso estos packages, $\widetilde{}$ no funciona a veces
\usepackage{graphicx}
\usepackage{times}

\newfont{\bb}{msbm10 at 11pt}
\newfont{\bbsmall}{msbm8 at 8pt}
\def\rth{\mathbb{R}^3}
\def\R{\mathbb{R}}

\def\N{\mathbb{N}}
\def\Z{\mathbb{Z}}

\def\esf{\mathbb{S}}

\newcommand{\ben}{\begin{enumerate}}
\newcommand{\bit}{\begin{itemize}}
\newcommand{\een}{\end{enumerate}}
\newcommand{\eit}{\end{itemize}}

\newcommand{\Int}{\mbox{Int}}
\renewcommand{\S}{\Sigma}

\newcommand{\wt}{\widetilde}
\newcommand{\HH}{\mbox{\bb H}}

\def\cL{{\cal L}}

\def\cS{{\cal S}}

\def\t{{\theta}}

\def\g{{\gamma}}
\def\G{{\Gamma}}
\def\l{{\lambda}}

\def\de{{\delta}}

\def\ve{{\varepsilon}}
\newcommand{\sol}{\mathrm{Sol}_3}

\def\centerbmp#1#2#3{\vskip#2\relax\centerline{\hbox to#1{\special
    {bmp:#3 x=#1, y=#2}\hfil}}}

\newtheorem{theorem}{Theorem}[section]

\newtheorem{remark}[theorem]{Remark}
\newtheorem{corollary}[theorem]{Corollary}
\newtheorem{definition}[theorem]{Definition}

\newtheorem{assertion}[theorem]{Assertion}

\newenvironment{proof}{\smallskip\noindent{\it Proof.}\hskip \labelsep}
{\hfill\penalty10000\raisebox{-.09em}{$\Box$}\par\medskip}

\begin{document}
\begin{title}{Finite topology minimal surfaces in homogeneous three-manifolds}
\end{title}
\vskip .2in

\begin{author}
{William H. Meeks III\thanks{This material is based upon
   work for the NSF under Award No. DMS-1309236.
   Any opinions, findings, and conclusions or recommendations
   expressed in this publication are those of the authors and do not
   necessarily reflect the views of the NSF.}
   \and Joaqu\'\i n P\' erez\thanks{The second
author was supported in part
by the MINECO/FEDER grant no. MTM2011-22547 and
MTM2014-52368.
}}
\end{author}
\maketitle
\begin{abstract}
We prove  that any complete,
embedded minimal surface $M$ with finite topology in a
homogeneous three-manifold $N$ has positive injectivity radius. When
one relaxes the condition that $N$ be homogeneous to that of
being locally homogeneous, then we show that the closure
of  $M$ has the structure of a minimal lamination of $N$.
As an application of this general result we prove that any complete,
embedded minimal surface with finite genus and a countable number of ends
is compact when the ambient space is $\esf^3$ equipped with  a homogeneous
metric of  nonnegative scalar curvature.


\end{abstract}
%\clearpage

%\tableofcontents

\section{Introduction.}
In this paper we apply the Local Picture Theorem on the Scale of Topology, which is Theorem~1.1 in~\cite{mpr14}
(see Theorem~\ref{tthm3introd} below for the statement of this result in the finite genus setting),
to prove that the injectivity radii of certain minimal surfaces in certain Riemannian three-manifolds are never zero.

 \begin{theorem}
\label{thm:finitetop}
Let $N$ be a complete, locally homogeneous three-manifold with positive injectivity radius.
Then, every complete, embedded minimal surface of finite topology in $N$ has positive injectivity radius.
\end{theorem}
In the case that the ambient three-manifold $N$ is isometric to $\esf^2\times \R$
with a scaling of its  standard product metric, then this result follows from
Theorem~15 in~\cite{mr13} where Meeks and Rosenberg also applied Theorem~1.1 in~\cite{mpr14}
to prove the stronger property that
that such minimal surfaces have bounded second fundamental form, linear ambient area growth
and so they are also proper in $\esf^2\times \R$.
For background material  on the geometry and classification
of homogeneous three-manifolds, see~\cite{mpe11}.

%The proof of this corollary will be given after the proof of Theorem \ref{thm:finitetop}.
The next result removes the positive injectivity radius assumption for the ambient space $N$.
The conclusion that we obtain in this setting is also weaker than the one in Theorem~\ref{thm:finitetop},
as follows from the Minimal Lamination Closure Theorem in~\cite{mr13}.

\begin{corollary}
\label{cor:ft}
If $M$ is a complete embedded minimal surface of finite
topology in a complete, locally homogeneous three-manifold $N$, then the closure
$\overline{M}$ has the structure of a minimal lamination of $N$. Furthermore:

\begin{enumerate}
\item Each limit leaf of $\overline{M}$ is stable (more precisely, the two-sided cover of the leaf is stable).
\item If $N$ has positive scalar curvature, then $M$ is proper in $N$.
\item If $N$ is simply connected and has nonnegative scalar curvature, then $M$ is proper in $N$.
\item If $N$ is the round three-sphere $\esf^3$, then $M$ is compact.
\end{enumerate}
\end{corollary}

\begin{remark}
  {\rm
Item 1 of Corollary~\ref{cor:ft} still holds without the hypothesis on $N$ to be locally homogeneous.
On the other hand, it can be shown that there exists a Riemannian metric of positive scalar curvature
on the three-sphere that admits a complete embedded minimal
plane whose closure does not admit the structure of a minimal lamination
(see e.g., \cite{CD2}). Hence, our hypothesis that $N$ is locally homogeneous is necessary
for items~{2, 3} of Corollary~\ref{cor:ft} to hold.
}
\end{remark}

In the sequel, we will denote by $B_M(p,r)$ (resp. $\overline{B}_M(p,r)$)
the open (resp. closed) metric ball centered at a point $p$
in a Riemannian manifold $M$, with radius $r>0$. In the case $M$ is complete,
we will let $I_M\colon M\to (0,\infty ]$ be
the injectivity radius function of $M$, and given a subdomain $\Omega \subset
M$, $I_{\Omega }=(I_M)|_{\Omega }$ will stand for the restriction of $I_M$ to $\Omega $.
The infimum of $I_M$ is called the {\it injectivity radius} of $M$.

We now briefly explain the main tool used in our proofs, namely Theorem~1.1 in~\cite{mpr14},
in the special case of surfaces of finite genus in a homogeneously
regular\footnote{A Riemannian three-manifold $N$ is {\it
homogeneously regular} if there exists an $\ve > 0$ such that
the image by the exponential map of any $\ve$-ball in a tangent space
$T_xN$, $x\in N$, is uniformly close to an $\ve$-ball in $\rth$ in
the $C^2$-norm. In particular, $N$ has positive injectivity radius. Note that
if $N$ is compact, then $N$ is
homogeneously regular.} three-manifold.
The paper~\cite{mpr14} was devoted to an analysis of the extrinsic geometry of
any embedded minimal surface $M$ (not necessarily with finite genus)
in small intrinsic balls in a homogeneously regular Riemannian
three-manifold, such that the injectivity radius function of $M$ is
sufficiently small in terms of the ambient geometry of the balls. We
carried out this analysis by blowing-up such an $M$ at a sequence of
points with {\it almost-minimal injectivity radius} (we will define this
notion precisely in items~{1, 2, 3} of the next theorem), which
produces a new sequence of minimal surfaces, a subsequence of which
has a natural limit object being either a properly embedded minimal
surface in $\rth$, a minimal parking garage structure on $\rth$
(see Section~3 of~\cite{mpr14}
for a discussion of this notion) or possibly, a
particular case of a singular minimal lamination of $\R^3$ with
restricted geometry; an important property is that this last possibility can
occur only if $M$ fails to have locally finite genus, in the sense given  by the
following definition.

\begin{definition}
{\rm
A Riemannian surface $M$ has {\em locally finite genus} if there exists an $\ve>0$ such
that intrinsic balls in $M$ of radius $\ve$ have uniformly bounded genus.
}
\end{definition}

The following result is an adaptation of Theorem~1.1 and Proposition~4.20 
in~\cite{mpr14} to the case of $M$ having locally finite genus.
\begin{theorem}[Local Picture on Scale of Topology for Locally Finite Genus]
\label{tthm3introd}
There exists a smooth decreasing function $\de\colon (0,\infty) \to (0,1/2)$
with $\lim_{r\to \infty} r\, \de(r)=\infty$
such that the following statements hold.
Suppose $M$ is a complete, embedded minimal
surface with injectivity radius zero and locally finite genus
in a homogeneously regular
 three-manifold~$N$. Then, there exists a sequence
of points $p_n\in M$ (called ``points of almost-minimal injectivity radius'')
and positive numbers $\ve _n= n\, I_{M}(p_n)\to 0$ such that:
%the following statements hold.
\begin{description}
\item[{\it 1.}]  For all $n$, the closure $M_n$ of the component
of $M\cap B_N(p_n,\ve _n)$  that contains $p_n$ is  a compact surface with boundary in
$\partial B_N(p_n,\ve _n)$. Furthermore, $M_n$ is contained in the intrinsic open ball
$B_M(p_n,\frac{r_n}{2}I_M(p_n))$, where
$r_n>0$ satisfies $r_n\, \de (r_n)=n$.
%In particular, $M_n$ is compact with boundary $\partial M_n\subset
%$M_n\subset B_M(p_n,\frac{\ve _n}{2\de })$

\item[{\it 2.}] Let $\l _n=1/I_{M}(p_n)$.
%, where $I_{M_n}$ denotes the injectivity radius function
%of $M$ restricted to $M_n$.
Then, %$ \ve_n\l_n=\infty$ %$\lim_{n\to \infty } \, \ve_n\l_n=\infty$
%and
$\l _nI_{M_n}\geq 1-\frac{1}{n}$ on $M_n$.

\item[{\it 3.}] The metric balls $\l_nB_N(p_n,\ve_n)$
of radius $n=\l_n\ve_n$
converge uniformly as $n\to \infty $ to $\rth$ with
its usual metric (so that we identify $p_n$ with
$\vec{0}$ for all $n$).
\end{description}
Furthermore, exactly one of the following two possibilities occurs.
\begin{description}
\item[{\it 4.}] The surfaces $\l_nM_n$ have
uniformly bounded Gaussian curvature
on compact subsets\footnote{As $M_n\subset B_N(p_n,\ve _n)$, the convergence
$\{ \l _nB_N(p_n,\ve _n)\} _n\to \R^3$ explained in item~{3} allows us to view
the rescaled surface $\l _nM_n$ as a subset of $\R^3$. The uniformly bounded
property for the Gaussian curvature of the induced metric on $M_n\subset N$ rescaled by
$\l _n$ on compact subsets of $\R^3$ now makes sense.} of
$\rth$ and there exists a connected, properly
embedded minimal surface $M_{\infty}\subset \R^3$
with $\vec{0}\in M_{\infty }$, $I_{M_{\infty}}\geq 1$
and $I_{M_{\infty}}(\vec{0})=1$,
%(here $I_{M_{\infty }}$ denotes the injectivity radius
%function of $M_{\infty }$)
 such that for any
$k \in \N$, the surfaces $\l _nM_n$ converge $C^k$
on compact subsets of $\rth$  to $M_{\infty}$ with
multiplicity one as $n \to \infty$.
\item[{\it 5.}] After a rotation in $\rth$, the surfaces $\l_nM_n$ converge to
a minimal parking garage structure on $\rth$,
consisting of a foliation $\cL$ of $\R^3$ by
horizontal planes, with two oppositely handed columns forming a
 set $S(\cL)$ of two vertical straight lines (the set $S(\cL)$ is the singular set of
 convergence of $\l _nM_n$ to $\cL$, see Definition~\ref{def2.2} below), and:
\begin{enumerate}[(5.1)]
\item $S(\cL)$ contains a line $l_1$  which
passes through the closed ball of radius 1 centered at the origin, and another line $l_2$ at
distance one from $l_1$.
\item \label{it5.2} There exist oriented closed geodesics  $\g_n\subset\l_nM_n$
with lengths converging to~$2$, which converge to the line segment $\g$ that joins
$(l_1\cup l_2)\cap \{x_3=0\}$ and such that the integrals of the unit conormal vector
of $\l_nM_n$ along $\g_n$ in the induced exponential $\rth$-coordinates of $\l _nB_N(p_n,\ve _n)$
converge to a horizontal vector of length $2$ orthogonal to $\g$.
\end{enumerate}
\end{description}
\end{theorem}
%The results in the series of papers
%\cite{cm21,cm22,cm24,cm23,cm35,cm25} by Colding and Minicozzi
%and the minimal lamination closure theorem by Meeks and
%Rosenberg~\cite{mr13}  play important roles in
%deriving the above compactness result. We conjecture that
%item~{6} in Theorem~\ref{tthm3introd} does not actually occur.
%
%A short explanation of the distribution of the paper is as follows.
%In Section~\ref{sec2} we introduce some notation and
%recall the notion and language of laminations, as well as a chord-arc
%property for embedded minimal disks previously proven by Meeks and Rosenberg
%and based on a similar one by Colding and Minicozzi. In Section~\ref{seclt2}
%we develop the theory of parking garage surfaces and limit
%parking garage structures, a notion that appears in item~5 of the
%main Theorem~\ref{tthm3introd}. Section~\ref{sec4}, the bulk of this paper,
%is devoted to the proof of Theorem~\ref{tthm3introd}. Sections~\ref{secap}
%and~\ref{secinj} include applications of Theorem~\ref{tthm3introd}; in particular,
%in Section~\ref{secinj} we apply the Local Picture Theorem on the Scale of Topology
%to study complete, embedded minimal surfaces of finite
%topology in a complete, locally homogeneous\footnote{A Riemannian
%manifold $N$ is {\it locally homogeneous} if given two points in
%$N$, balls of the same sufficiently small radius centered at these points are
%isometric. } three-manifold~$N$. We
%refer the reader to~\cite{mpe3,mpr8,mpr9,mpr11,mr13,mt13,mt9} for further
%applications of Theorem~\ref{tthm3introd}.

\begin{definition}
\label{def2.2}
{\rm
If $\{ \S_n\} _n$ is a sequence of complete embedded minimal surfaces in a Riemannian three-manifold
$N$, consider the closed set
$A\subset N$ of points $p\in N$ such that for every neighborhood $U_p$ of $p$ and every subsequence
of $\{ \S_{n(k)}\} _k$, the sequence of norms of the second fundamental forms of $\S_{n(k)}\cap U_p$
is not uniformly bounded.
By the arguments in Lemma~1.1 of Meeks and Rosenberg~\cite{mr8},
after extracting a subsequence, the $\S_n$ converge
on compact subsets of $N-A$ to a minimal lamination $\cL'$ of $N-A$ that extends to a minimal lamination
$\cL$ of $N-\cS$,
where $\cS\subset A$ is the (possibly empty) {\it singular set} of $\cL$, i.e., $\cS$ is
the closed subset of $N$
such that $\cL$ does not admit a local lamination structure around any point of $\cS$.
We will denote by $S(\cL)=A-\cS$ the {\it singular set of convergence} of the $\S_n$
to $\cL$, i.e., those points of $N$ around which $\cL$ admits a lamination structure
but where the second fundamental forms of the $\S_n$ still blow-up.
}
\end{definition}

\section{Proof of Theorem \ref{thm:finitetop} and Corollary~\ref{cor:ft}.}
\noindent {\em Proof of Theorem \ref{thm:finitetop}.}
Let $M$ be a complete, embedded minimal surface of finite topology in a complete, locally homogeneous
three-manifold $N$ with positive injectivity radius. Since $M$ has finite
topology, then  it has a finite number of ends,
all being topologically annuli. If the injectivity radius function $I_M$ of $M$
is bounded away from zero on each of its annular ends, then $I_M$ is
globally bounded away from zero. Hence, the theorem will follow provided that we show that
$I_M$ is bounded away from zero on each end of $M$, or equivalently,
if each end representative $E$ of $M$ satisfies the following property:
\begin{quote}
(End) If $f\colon E=\esf^1\times [0,\infty)\to N$
is a complete injective minimal immersion and $N$ satisfies the hypotheses of
Theorem~\ref{thm:finitetop},
then the injectivity radius function $I_E$ of $E$ is bounded away from zero
in the complement of any neighborhood of the boundary $\partial E$.
\end{quote}
Our strategy to prove property (End) will be first prove it in  the particular case
 when $N$ is simply connected (Assertion~\ref{ass} below) and then use this
 particular case to demonstrate the general case
 (Assertion~\ref{ass2}).

\begin{assertion} \label{ass}
If $N$ is simply connected, then property {\rm (End)} holds.
\end{assertion}
\begin{proof}
As $N$ is locally homogeneous, complete, simply connected and has dimension three, then $N$ is homogeneous.
In this setting and as shown in~\cite{mpe11}, $N$ is isometric to either $\esf ^2\times \R $ 
with a scaling of its standard metric
or to a three-dimensional metric Lie group, i.e., a Lie group endowed with a left invariant metric.
If $N=\esf^2\times \R $ with a scaling of its standard metric, then
we refer the reader to~\cite{mr13} that contains the stronger result
that the second fundamental form of a complete annular minimal end $E$ is bounded;
also see the discussion just after the statement of Theorem~\ref{thm:finitetop}.
Therefore, in the sequel  we will assume that
$N$ is a metric Lie group.

 For simplicity, we let $E=f(E)$ denote the embedded minimal annulus.
If $I_E$ is not bounded away from zero outside of some neighborhood of $\partial E$,
then there exists an intrinsically divergent sequence of points $p_n\in E$, with
$I_E(p_n)<\frac{1}{n}$, $d_E(p_n,p_{n+1})>n$; such a sequence of points $p_n$
can be chosen to be points of almost-minimal injectivity radius.
By Theorem~\ref{tthm3introd}, after choosing a subsequence, the local picture on the scale of topology
of $E$ around the sequence $p_n$ is either a minimal parking garage structure of $\rth$ with
two oppositely handed columns (that is, item~5 of Theorem~\ref{tthm3introd} occurs),
or a properly embedded minimal surface $M_{\infty }$ with genus zero
(item~4 of Theorem~\ref{tthm3introd} occurs).
In this last case, $M_{\infty }$ is a catenoid
or a Riemann minimal example by classification results (see~\cite{mpr6}
and references therein).

We claim that the sequence of points $p_n$ is diverging in $N$.
Arguing by contradiction, we may assume, after replacing by a subsequence,
that the points $p_n$ converge to a point $p\in N$. Using normal coordinates
around $p$ and taking a further subsequence we may also assume that the rescaled
space in which the local picture of $E$ exists is $\rth$ with induced coordinates,
and that the flux vector of the rescaled limit object is not zero and
parallel to $(0,0,1)$; see item~\ref{it5.2} of Theorem~\ref{tthm3introd}
for the definition of this flux vector in the case that the limit 
object is a parking garage structure on $\rth$, i.e., it is not a catenoid or a 
Riemann minimal example in which cases the flux vector is just the flux vector 
for any of the circles on the limit surface.
Let $V$ be the right  invariant vector field on $N$
determined by $V(p)=(0,0,1)$ (with respect to a previously chosen orthonormal basis
of $T_pN$ so that under rescaling of the metric, the direction of $V$ at $p$ gives rise
to the vertical direction in the limit $\R^3$). Note that since the metric on
$N$ is left invariant, then every right  invariant vector field on $N$ is a Killing
field. Let $V^T$ denote the tangential part of $V$ on $E$. Consider the (scalar) flux of $V^T$
across any oriented closed curve $\G$ in $E$, defined as
\begin{equation}
\label{flux}
\mbox{Flux}(V^T,\G )=\int _{\G }\langle V,\eta \rangle ,
\end{equation}
where $\langle ,\rangle $ stands for the ambient metric on $N$ and $\eta $ is a unit conormal to $E$ along $\G $.

Since
the mean curvature of $E$ is zero and  $V$ is a Killing field, then the divergence of
$V^T$ in $E$ vanishes identically, and an elementary application of
the divergence theorem implies that Flux$(V^T,\G )$ is a homological invariant of
$\G$. Choose simple closed oriented curves $\G_n\subset E$ near $p_n$ for $n$ large, so that
$\mbox{Flux}(V^T,\Gamma _n)\neq 0$, which can be done since after
rescaling, the vertical component of the flux vector of the limit object is not zero.
As the homology group $H_1(E)$ is $\Z $, we deduce that the $\G_n$ represent one of the
two nontrivial generators
of $H_1(E)$ and that the $\G_n$ can be oriented so that
$\mbox{Flux}(V^T,\Gamma _n)=\mbox{Flux}(V^T,\partial E)$ for all $n$.
But by the injectivity radius zero assumption on $E$, these curves $\G _n$
can be chosen to have lengths converging to zero,
which is impossible because the $\G_n$ converge
to $p$ and $|V(p)|=1$ and so $|\mbox{Flux}(V^T,\G_n)|\leq 2\mbox{ length}(\G_n)$.
This contradiction implies that the blow-up points $p_n$ are divergent in $N$.

The proof of Assertion \ref{ass} now breaks up into 3 cases, depending on the ambient space.

\begin{description}
%\item[Case A:] \; $N$ is isometric some $\esf^2\times \R$.
\item[Case A:]\; $N$ is the  special unitary
Lie group $SU(2)$ with a left invariant metric.
\item[Case B:] $N$ is the Lie group $\widetilde{\mbox{PSL}}(2,\R)$ (the universal cover
of the group of isometries of the hyperbolic plane) with a left invariant metric.
\item[Case C:]\; $N$ is a semidirect product $\R^2\rtimes _A\R$ for some real
$2\times 2$ matrix $A$, equipped with a left invariant metric; %see Section~2.2
%of~\cite{mmpr4}
see Section~2.1 in~\cite{mpe11} for details.
This case includes all the noncompact, simply connected, nonunimodular metric Lie groups
as well as the Heisenberg group Nil$_3$,
the solvable group $\sol$, the universal cover $\widetilde{E}(2)$ of the Euclidean group of
isometries of $\R^2$ and the abelian Lie group $\rth$, each of them equipped
     with any left invariant metric.
\end{description}
\par
\noindent
{\bf Proof of Case A}: This case follows immediately from the fact that  $SU(2)$ is compact
and the just proved fact that whenever Assertion \ref{ass} fails,
there exists a sequence of  points $p_n\in E$ of almost-minimal injectivity radius
that diverges in the ambient space.
\par
\vspace{.2cm}
\noindent
{\bf Proof of Case B}:  Let $G\colon E\to \esf^2$ denote the left invariant Gauss map for $E$;
by this we mean the mapping valued in the unit sphere of the tangent
space $T_eN$ to $N$ at the identity element $e\in N$ (or equivalently, the Lie algebra $\mathfrak{g}$ of
the metric Lie group $N$),
which assigns to each point $p\in E$ the left translation $G(p)\in T_eN$
of the unit normal vector field $\nu _p$ to $E$ at $p$, i.e., $(l_p)_*G(p)=\nu _p$, where
$l_p\colon N\to N$ denotes the left translation by $p$ and $(l_p)_*$ is its differential.
As explained in the second paragraph of  the proof of the assertion,
after rescaling $E$ on the scale of topology around a sequence of points $p_n\in E$ of almost-minimal
injectivity radius,
we find a limit object which is a catenoid, a Riemann minimal example or a minimal parking garage
structure in $\R^3$ with two oppositely handed columns. We will treat each of these three cases
separately.

 {\bf After choosing a subsequence,
suppose that the local picture of $E$ for the sequence $p_n$ is a catenoid.}
Since
in the blow-up process we rescale the ambient metric, then after choosing an
orthonormal basis for the metric Lie algebra of $N$, we can consider the ambient tangent directions
as fixed and so, it makes sense to consider for $n$ large, points $q_n\in E$ arbitrarily close to $p_n$ so that
the unit vectors $\nu _{q_n}\in T_{q_n}N$ converge as $n\to \infty $ to one of the two limiting normal
vectors to the parallel ends of the limit catenoid.
After choosing a subsequence, we can assume that the left translated vectors $G(q_n)=(l_{q_n}^{-1})_*\nu _{q_n}\in
\esf^2 \subset T_e N$ converge as $n\to \infty $ to a vector $w\in \esf^2$.

We next briefly describe the set of two-dimensional subgroups of $\widetilde{\mbox{PSL}}(2,\R )$ as we will use
them in the argument; for details on the following description, see Section~2.7 in~\cite{mpe11}.
Recall that the projective special linear group PSL$(2,\R)$ can be considered to be the group Iso$^+(\HH^2)$
of  orientation-preserving isometries of
the hyperbolic plane $\HH^2$. Using the Poincar\'e disk model for $\HH ^2$, for any point
$\theta \in \esf^1=\partial _{\infty }\HH^2$ one may consider the subset $\HH _{\theta }$ of PSL$(2,\R)$ given by
\[
\HH _{\t }=\{ \phi \in \mbox{Iso}^+(\HH^2) \ | \  \phi (\theta )=\t \} .
\]
$\HH_{\theta }$
is a connected two-dimensional subgroup of PSL$(2,\R)$, %and it is
generated by the hyperbolic translations along geodesics
of $\HH^2$ one of whose ends points is $\t $ and the parabolic
translations along horocycles tangent at $\t $.
%$\HH_{\t }$ is a noncommutative group, isomorphic to $\HH^2$.
Consider the covering map $\Pi\colon N=\widetilde{\mbox{\rm PSL}} (2,\R)\to \mbox{PSL}(2,\R)$, which is a
group homomorphism.

Let $\widetilde{\HH}_{\theta}\subset \Pi^{-1}(\HH_{\theta})$ be the connected, two-dimensional
subgroup of $N$ passing through the identity.
%Since $\HH _{\t}$ is simply connected, then $\widetilde{\HH_{\t }}$
%is isomorphic to $\HH_{\t }$.
Since by definition, the left invariant Gauss map of a surface
is invariant under left translations, it is clear that the left invariant Gauss map
of a two-dimensional subgroup is constant. Particularizing to the case of the circle family
of subgroups $\{ \widetilde{\HH}_{\theta}\ | \ \t \in \esf^1\} $ of $N$, each $\widetilde{\HH}_{\theta}$
has constant left invariant Gauss map $\G (\t )\in \esf^2\subset T_e\widetilde{\mbox{\rm PSL}}(2,\R )$.
Note that $\t \in \esf ^1\mapsto \G (\t )\in \esf^2$ is injective, because
the Gauss map image of a two-dimensional subgroup determines the subgroup itself.

Choose one of these two-dimensional subgroups $\wt{\HH}=\wt{\HH}_{\theta_0}$ so
that the normal vector $\G(\t _0)$ is different from $\pm w$.
As the ambient metric of $\widetilde{\mbox{\rm PSL}}(2,\R )$ is left invariant,
it follows that the family of left cosets of $\wt{\HH}$ forms a codimension-one 
foliation of $\widetilde{\mbox{\rm PSL}}(2,\R )$, all whose leaves are ambiently isometric to $\wt{\HH}$
and have the same constant value $\G (\t _0)$ for their left invariant Gauss map as $\wt{\HH}$.
After choosing a subsequence of $p_n$, Theorem~\ref{tthm3introd} implies that
we can find a sequence $\ve _n>0$ converging to zero
such that
for each $n\in \N$, the closure $C_n$ of the component of $E\cap B_N(p_n,\ve _n)$ that contains $p_n$
is a compact annulus and if $n$ is large enough, $C_n$ is arbitrarily close to large compact piece $C_{\infty }(n)$
of a complete catenoid $M_{\infty }\subset \R^3$ with $\vec{0}\in M_{\infty }$, $I_{M_{\infty }}\geq 1$ and
$I_{M_{\infty }}(\vec{0})=1$. In particular, $\vec{0}$ lies in the waist circle of $M_{\infty }$.
Without loss of generality, we can assume that $C_{\infty }(n)$ is obtained by intersecting
$M_{\infty }$ with a ball of radius $n=\l _n\ve _n$ in $\R^3$ centered at $\vec{0}$ (here $\l _n=
1/I_E(p_n)$, see items~2 and 3 of Theorem~\ref{tthm3introd}). Take $r>0$ so that
the waist circle $\g _{\infty }$ of $M_{\infty }$ lies in the ball of radius $r/2$ centered at $\vec{0}$.
For $n$ large, there exists a simple closed geodesic $\g_n\subset C_n$
such that $\l _n\g _n$ converges as $n\to \infty $ to $\g _{\infty }$.
In particular for $n$ large, $\g _n$ lies in the ball $B_N(p_n,\de _n)$, where $\de _n=r/\l _n$  (hence $0<\de _n<\ve _n$).
Given $n\in \N$, let
 \[
 U_n=\{ q\in \widetilde{\mbox{\rm PSL}}(2,\R ) \ | \ \mbox{dist}(q,
 p_n\wt{\HH})<\de _n\}
 \]
 be the open tubular neighborhood of the topological plane
 $p_n\wt{\HH}$ of radius $\de _n$, where dist denotes extrinsic
 distance in $\widetilde{\mbox{\rm PSL}}(2,\R )$. Therefore the following property holds for all $n$ large:
\begin{enumerate}[(P1)]
\item $\g _n$ lies inside $B_N(p_,\de _n)\subset U_n$.
\end{enumerate}
By Lemma~3.9 in~\cite{mpe11}, given a two-dimensional subgroup
$H$ of $\widetilde{\mbox{\rm PSL}}(2,\R )$, the family
${\mathcal F}(H)=\{ Hx\ | \ x\in N\} $ of right cosets
of $H$ forms a codimension-one foliation of
$\widetilde{\mbox{\rm PSL}}(2,\R )$ whose leaves have the property
that each one is at a constant distance from each other. If we particularize
to $H=H_n:=p_n\wt{\HH}p_n^{-1}$, it follows that each of
the leaves of ${\mathcal F}(H_n)$ is at a constant
distance from $H_n p_n=p_n\wt{\HH}$. This implies that
 $U_n$ is the union of the right cosets of $H_n$ at
distance less than $\de _n$ from $p_n\wt{\HH}$. Another consequence of
this description is that $N-U_n$ is the set of points in $N$
at distance at least $\de_n$ from $p_n\wt{\HH}$. Therefore,
if we choose $n$ large enough, the following property holds:
\begin{enumerate}[(P2)]
\item Each of the two boundary curves of $C_n$ intersects the two
components of $N-U_n$.
\end{enumerate}

Let $E_n\subset E$ be the subannulus with boundary $\g_1\cup \g_n$.
%By Lemma~3.9 in~\cite{mpe11}, $N$ is foliated by the right cosets
%$\{ H_1x\ | \ x\in N\} $ of the subgroup $H_1=p_1 \wt{\HH}p_1^{-1}$,
% and these cosets have the property that each one is at a constant
% distance from the right coset $H_1 p_1=p_1\wt{\HH}$,
%and at constant distances from each other.
Choose an integer $k\in \N$ sufficiently large,  so that
$\de_k<\de_1$.
Since the two boundary components of $U_1$ %$U(p_1\wt{\HH },\ve _1)$
 are each right cosets of $H_1$ of  constant
distance $2\de_1$ from each other, then 
the triangle inequality and  property (P1) imply that
the boundary curve $\g_k$ of $E_k$
does not intersect both of the boundary components $\partial_1, \partial_2$ of $U_1$;
%$U(p_1\wt{\HH },\ve _1)$;
as $\g_1$ lies in between $\partial_1, \partial_2$, then
we deduce that $\g_k$ is at the same side of at least one of 
the two components $\partial_1, \partial_2$ as $\g_1$,
say this component is $\partial _1$.
Property (P2) applied to $C_1$ gives that $E_k$ contains points in the 
component $\Delta$  of $N-\partial_1$ that is disjoint from the 
boundary of $E_k$. By the compactness of $E_k$, there exists a point
$q_k\in \Delta \cap E_k$ furthest from $\partial _1$.
By the above description of the leaves of ${\mathcal F}
(H_1)$ as level sets of the distance function to $\partial _1$, we
deduce that
the right coset $H_1 q_k\subset \Delta $
lies on one side of $E_k$ at $q_k$.  Since
$\wt{\mbox{PSL}}(2,\R)$ is a unimodular Lie group, then Corollary~3.17 in~\cite{mpe11}
implies that $H_1q_k$ is a minimal surface, which gives a contradiction
by the maximum principle for minimal surfaces applied at the point $q_k$, see Figure~\ref{fig3}.
\begin{figure}
\begin{center}
\includegraphics[width=9.5cm]{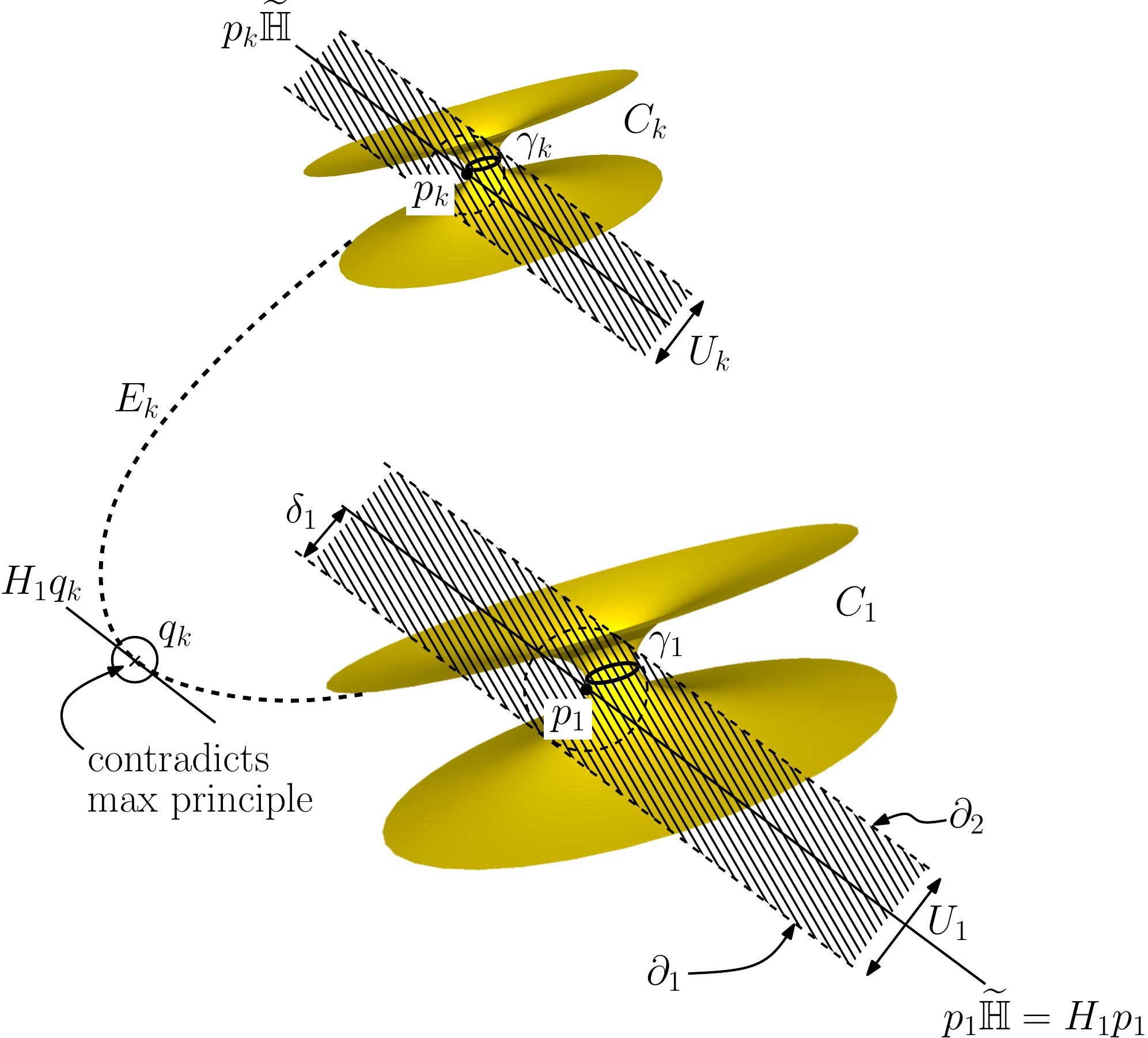}
\caption{The boundary (topological) planes
$\partial _1,\partial _2$ of
$U_1$ are right cosets of $H_1$ at constant
distance $\de _1$ from $H_1p_1$. The contradiction comes from application of the
maximum principle to the minimal surfaces $H_1q_k$ and $E_k$ at an
appropriately chosen interior point $q_k$ of $E_k$.}
\label{fig3}
\end{center}
\end{figure}
This completes the proof of Case~B
when the local picture associated to the sequence $p_n$ is a catenoid.

The proof of
Case~B when the associated local picture is a Riemann minimal
example or a parking garage structure on $\rth$
is essentially identical to the case of the local picture
being a catenoid; in these two subcases, after choosing
a subsequence, there are  associated limit normal vectors $\pm w$
for the related ends of the Riemann minimal example or the planes
of the foliation in the limit
parking garage structure, just as was the case when the local
picture was a catenoid.  Once this choice of subsequence
is made, the proof follows the same reasoning as in the catenoid case
to obtain a contradiction, where in the case of a parking garage structure one
uses the curves $\g_k$ given in item~5.2 of Theorem~\ref{tthm3introd}.  This completes
the proof of the assertion when Case~B holds.
\par
\vspace{.2cm}
\noindent
{\bf Proof of Case C}: Suppose $N$ is a semidirect product $\R^2\rtimes _A\R$
with a left invariant metric, for some real $2\times 2$ matrix $A$.
As in the previous Case~B, we will give details when the local picture
associated to the sequence of points $p_n$ of almost-minimal injectivity radius is a catenoid; the
cases where this local picture is a Riemann minimal
example or a parking garage structure can be argued
similarly and we leave the details to the reader.

Assume that the local picture associated to the blow-up points $p_n$ is a catenoid.
Without loss of generality, we may assume that the
constant mean curvature of the horizontal planes
$x_3^{-1}(t)=\R^2 \rtimes_A\{t\}$ in $N$ is trace$(A)/2\geq 0$ with respect to their
unit normal vector field $E_3=\frac{\partial}{\partial  x_3}$, which is left invariant
(see Section 2.3 of~\cite{mpe11}).
By Remark~2.10 in~\cite{mpe11}, the level sets of the coordinate
functions $x_1, x_2$ are minimal planes in $N$.
It follows from the mean curvature comparison principle
that if the $x_3$-coordinate function of $E$
has a local minimum at a point  $q\in \Int(E)$,
then $E\subset \R^2 \rtimes_A\{x_3(q)\}$.  Similarly,
if the $x_i$-coordinate function, $i=1,2$, of
$E$ has a local minimum or maximum
at a  point  $q\in \Int(E)$,
then $E\subset x_i^{-1}(x_i(q))$.

As in the proof of Case~B, we may assume that after choosing a subsequence,
the left invariant Gauss map of the ends of the limit local
picture catenoid of $E$ is $\pm w\in \esf^2$.  First
consider the case where $w\neq \pm E_3(e)$.
After choosing a subsequence of $p_n$, we can find a sequence $\ve _n>0$ converging to zero
such that the compact catenoidal pieces
$C_n\subset E\cap B_N(p_n,\ve _n)$,
which converge after blowing-up on the scale of
topology to a large compact piece of a catenoid in $\R^3$ containing its waist circle,
satisfy the following properties:
\begin{enumerate}[(P1)$'$]
\item The almost-waist circle (simple closed
geodesic) $\g _n$ of $C_n$ lies inside
the open tubular neighborhood 
$U_n=\R^2 \rtimes _A\{ x_3\in\R \mid \, |x_3-x_3(p_n)|<\de_n\}$ of radius
$\de _n=r/\l _n$
of the horizontal plane $\R^2 \rtimes_A \{x_3(p_n)\}$ and 
inside $B_N(p_n,\de_n)$ (here $\l _n=1/I_E(p_n)$ and $r>0$ are chosen
as in the proof of Case~B).
\item Each of the two boundary curves of $C_n$ intersects the two components
of $N-U_n$.
\end{enumerate}

Given $n\in \N$, let $E_n\subset E$ be the subannulus with boundary $\g_1\cup \g_n$.
Then arguing as in  the proof of Case~B, one finds that for $k$ large,
the $x_3$-coordinate function of $E_k$ has its minimum value at
an interior point $q_k$ of $E_k$,
which forces $E$ to be an end representative of
the plane $\{x_3= x_3(q_k)\}$. Since this plane is isometric to
the flat $\R^2$ but the injectivity radius function of $E$ was assumed to take values
arbitrarily close to zero away from its boundary, we obtain a contradiction.
Henceforth, we will  assume that $w=\pm E_3(e)$.

Since the local picture of $E$ associated to the blow-up points $p_n$ is a catenoid with a vertical axis,
then it is clear that the $x_1$-coordinate of
a similarly defined subannulus $E_k$ cannot have its global maximum at the boundary of $E_k$.
Therefore, $E$ must be contained in a vertical plane of the form $\{ (t_0,x_2,x_3) \mid x_2,x_3\in \R\}$
for some $t_0\in \R$. Exchanging $x_1$ by $x_2$ we find that $E$ is contained in a plane of the form
$\{ (x_1,t_0',x_3) \mid x_1,x_3\in \R\}$ for some $t_0'\in \R $, which is a contradiction.
This completes the proof of Assertion~\ref{ass}. \end{proof}

\begin{assertion} \label{ass2}
Property {\rm (End)} holds.
\end{assertion}
\begin{proof}
Arguing by contradiction, assume that $p_n\in E$ is a divergent sequence in $E$ with
$I_E(p_n)<\frac{1}{n}$. As said in the second paragraph of the proof of Assertion~\ref{ass},
the points $p_n$ can be assumed to be points of almost-minimal injectivity radius.
Let $\l _n=1/I_E(p_n)$. Since the scaled surfaces $\l_n M_n$ are minimal in $\l _nN$ and the sectional curvatures of the
ambient spaces $\l_n N$ are converging uniformly to zero,
then the Gauss equation implies that the exponential map $\exp _{p_n}$ of $ T_{p_n}(\l_n M_n)$
restricted to the closed metric ball of radius $2$ centered at the origin
is a local diffeomorphism for $n$ sufficiently large. As the injectivity radius of $\l _nM_n$ at $p_n$ is 1, then
$\exp _{p_n}$ is a diffeomorphism when restricted to the open disk of radius $1$, and it fails
to be injective on the boundary circle of radius $1$.
Now it is standard to deduce the existence of simple closed geodesic loops
$\G_n$ in $E$ based at $p_n$, each  of length less than $\frac2n$.

%(note that $N$ is homogeneous, then the Gaussian curvature of $E$ is bounded from above).
Let $\Pi\colon \widetilde{N}\to N$ be the universal cover of $N$.
%, and consider a connected component $\wt{E}$ of $\Pi^{-1}(E)$ in $\wt{N}$.
We discuss two possibilities, depending on whether or not the geodesic loop $\G_n$ is homotopically trivial
in $E$ (note that after extracting a subsequence, we can assume all the $\G_n$ satisfy exactly one of the
possibilities below).
\begin{enumerate}[(Q1)]
\item If all the $\G_n$ are homotopically nontrivial
in $E$, then for $n$ large $f$ lifts through $\Pi $ to a complete embedding
$\widetilde{f}\colon E\to \widetilde{N}$, once we have picked an initial point in
$\Pi ^{-1}(f(\G _n))$ (because as the lengths of the curves $f(\G _n)$ tend to zero as $n\to \infty $, then
for $n$ large the image loops $f(\G_n)$ lie in extrinsic balls in $N$ of arbitrarily small radius
and these extrinsic balls are topological balls for $n$ sufficiently large
because $N$ has positive injectivity radius;
hence, the induced map by $f$ on the fundamental groups is trivial).
In this setting, Assertion~\ref{ass} gives a contradiction.

\item All the $\G_n$ bound disks $D_n$ in $E$. In this case, we consider a
connected component $\widetilde{E}$ of $\Pi^{-1}(E)$ in $\widetilde{N}$. Observe that the disks
$D_n$ lift through $\Pi $ to a sequence of related disks $\widetilde{D}_n\subset \widetilde{E}$.
It follows from the arguments in the proof of
Theorem~\ref{tthm3introd} applied to $\widetilde{E}$ at the sequence of
preimages $\wt{p}_n\in \partial \wt{D}_n$ of the $p_n$ (the $\wt{p}_n$ are points
of almost-minimal injectivity radius for $\widetilde{E}$) that the related local picture
on the scale of topology of
$\widetilde{E}$ around the $\wt{p}_n$ is either a catenoid, a Riemann minimal example
or a minimal parking garage structure with two oppositely handed columns.
As observed at the beginning of the proof of Assertion~\ref{ass},
this implies that the scalar flux Flux($K_n^T,\partial \wt{D}_n)$ with respect to certain
Killing fields $K_n$ in $\wt{N}$, defined as in (\ref{flux}), is nonzero.
As these fluxes are homological
invariants associated to $\partial \wt{D}_n$ but each $\partial \wt{D}_n$ is
homologically trivial on $\wt{E}$, we obtain a contradiction.
\end{enumerate}
\par
\vspace{-.8cm}
\end{proof}

As explained at the beginning of the proof of Theorem~\ref{thm:finitetop},
the validity of Property (End) finishes the proof of the theorem.
{\hfill\penalty10000\raisebox{-.09em}{$\Box$}\par\medskip}

\vspace{.4cm}

\noindent {\em Proof of Corollary~\ref{cor:ft}.} Let $M$ be a complete embedded minimal surface
of finite topology in a complete, locally homogeneous three-manifold $N$. In order to conclude
that $\overline{M}$ has the structure of a minimal lamination, it suffices to check that the
injectivity radius function of $M$ is bounded away from zero in every extrinsic ball of $N$, see
Remark~2 in~\cite{mr13}. This extrinsic local positivity of the injectivity radius function
of $M$ follows directly from modifications of the lifting arguments in the proof of Assertion~\ref{ass2};
the only modification occurs in case (Q1) above, where we must replace
the hypothesis that the injectivity radius of $N$ is positive by the property that
the injectivity radius function $I_N$ in any compact extrinsic ball of $N$ is bounded away from zero.
Therefore, $\overline{M}$ is a minimal lamination of $N$.

By the Stable Limit Leaf Theorem for $H$-laminations in~\cite{mpr19,mpr18}, the limit leaves of $\overline{M}$
have the property that their two-sided covers are stable.  This proves item~{1} of the corollary.

Assume now that $N$ has positive scalar curvature (as $N$ is locally homogeneous, then its scalar curvature is constant).
If $M$ is not proper in $N$, then $\overline{M}$ contains a
limit leaf $L$. Since the two-sided cover of $L$ is stable, it follows from item~1 of Theorem~2.13
in~\cite{mpr19} that $L$ is either an embedded sphere or an embedded projective plane.
After lifting the lamination $\overline{M}$ to a two-sheeted cover of $N$, we may assume that
$L$ is a sphere. But this is impossible since it is a leaf of a lamination and so the nearby leaves must also be
spheres which is not true since $M$ is noncompact. This contradiction proves item~{2} in the corollary.

To prove item 3 of Corollary~\ref{cor:ft}, assume
that $N$ has nonnegative scalar curvature and it is simply connected.  By item~2
of the corollary, we may assume that the scalar curvature of $N$ vanishes.
Since $N$ is homogeneous, then,
metric classification results in Section~4 of Milnor~\cite{mil2} imply that $N$ is either isometric to $\R^3$
 or $N$ is $SU(2)$ with a left invariant metric.
By  the main result of Colding and Minicozzi in~\cite{cm35}, if $N$ is flat, then $M$ is proper
(in fact, $M$ has positive injectivity radius by the results
in~\cite{mr13}).

To finish the proof of item~3 of the corollary, it remains to demonstrate that
if $N$ is $SU(2)$ endowed with a left invariant metric of zero scalar curvature, then
$M$ is proper. We will show that $M$ is indeed compact, even in the more general
case that the constant scalar curvature of the left invariant metric on $SU(2)$
is nonnegative, thereby proving also item~4 of Corollary~\ref{cor:ft} (this will
then finish the proof of the corollary).

Suppose that $M$ is noncompact and we will derive a
contradiction.  As $M$ is noncompact, then the minimal lamination $\overline{M}$
has a  limit leaf $L$ whose two-sided cover $\wt{L}$ must be stable. Since the scalar
curvature of $N$ is assumed to be nonnegative, then item~2(a) of Theorem~2.13
in~\cite{mpr19} implies that  $\wt{L}$ has at most quadratic area growth.
In this situation, it follows from Theorem~1 in~\cite{mper1} that the
space of bounded Jacobi functions on $\wt{L}$ is one-dimensional, and this space coincides
with the space of Jacobi functions with constant sign on $\wt{L}$.
Consider the three-dimensional space of right invariant Killing fields on $N$.
As $N$ is compact, then every such vector field $F$ is bounded on $N$, and thus, the
inner product $\langle F,\nu \rangle $ of $F$ with the unit normal vector field $\nu $ to $\wt{L}$
defines a bounded Jacobi function on $\wt{L}$. Now pick a point
$p\in \wt{L}$ and let $F_1,F_2$ be linearly independent right invariant vector fields
on $SU(2)$ so that $F_1(p),F_2(p)$ are tangent to $\wt{L}$ at $p$. Since for $i=1,2$
the function $u_i=\langle F_i,\nu \rangle $ is a bounded Jacobi function on
$\wt{L}$, then $u_i$ has constant sign on $\wt{L}$. As $u_i$ vanishes at $p$, then
$F_i$ is everywhere tangent to $\wt{L}$, for $i=1,2$. In particular, the Lie bracket
$[F_1,F_2]$ is also everywhere tangent to $\wt{L}$. This is impossible, as $[F_1,F_2]$ is
a right invariant vector field on $SU(2)$ which is everywhere linearly independent with
$F_1,F_2$ (this last property follows from the structure constants of the unimodular
group $SU(2)$).
This contradiction implies that $M$ must be compact, which completes the proof of Corollary~\ref{cor:ft}.
{\hfill\penalty10000\raisebox{-.09em}{$\Box$}\par\medskip}

\begin{remark} \label{remarkAbs}{\em  We now prove the result stated 
in the last sentence of the abstract.
To see this property holds first observe that,
by the arguments at the end of the proof of Corollary~\ref{cor:ft},
 $N=\esf^3$ equipped with a homogeneous metric with
nonnegative scalar curvature satisfies the hypothesis of $N$ described in the
next corollary, which means that any complete embedded minimal surface in $N$
cannot have an annular end. Next observe that if $\S$ is a surface of finite genus with
countably many ends, then $\S$ is compact or it must have at least one annular end,
and so, the stated property in the abstract now follows.}
\end{remark}

\begin{corollary}
Let $N$ be a complete, locally homogeneous three-manifold which does not admit
any complete embedded stable minimal surfaces. If $M$ is a complete embedded minimal
surface in $N$, then every annular end of $M$ is proper.

\end{corollary}

\begin{proof}
Suppose $E\subset M$ is an annular end.  With minor modifications,
the proof of Corollary~\ref{cor:ft} shows that the
injectivity radius function of $M$ restricted to $E$
is bounded away from zero on compact subdomains of $N$.
Suppose that $E$ is not proper. Then, the limit set Lim$(E)$ of $E$ is nonempty.
As the injectivity radius function of $M$ restricted to $E$ is
bounded away from zero on compact subdomains of $N$, the Lamination 
Closure Theorem in~\cite{mr13} implies that Lim$(E)$ has the structure 
of a minimal lamination of $N$. By the Stable Limit Leaf Theorem 
for $H$-laminations~\cite{mpr19,mpr18}, the two-sided cover
of any leaf of Lim$(E)$ is stable, thereby contradicting our hypotheses on $N$.
\end{proof}

\vspace{.2cm}

\center{William H. Meeks, III at profmeeks@gmail.com\\
Mathematics Department, University of Massachusetts, Amherst, MA 01003}
\center{Joaqu\'\i n P\'{e}rez at jperez@ugr.es\\
Department of Geometry and Topology and Institute of Mathematics
(IEMath-GR), University of Granada, 18071, Granada, Spain}

\bibliographystyle{plain}
\bibliography{bill}
\end{document}